\documentclass{article}
\usepackage{amsfonts}
\usepackage{amsmath}
\usepackage{amsthm}
\usepackage{tikz-cd}
\usepackage{hyperref}
\usepackage{cleveref}
\usepackage{authblk}
\hypersetup{colorlinks=false}
\usepackage[style=numeric]{biblatex}
\newtheorem{theorem}{Theorem}
\newtheorem{lemma}[theorem]{Lemma}
\newtheorem{corollary}[theorem]{Corollary}
\newtheorem{proposition}[theorem]{Proposition}
\newtheorem*{conjecture}{Conjecture}
\newtheorem*{remark}{Remark}
\crefname{theorem}{Theorem}{Theorems}
\crefname{lemma}{Lemma}{Lemmas}
\crefname{corollary}{Corollary}{Corollaries}
\crefname{proposition}{Proposition}{Propositions}
\crefname{conjecture}{Conjecture}{Conjectures}

\title{Ramified cover of varieties with nef cotangent bundle}
\author{
    {\sc Yiyu Wang}\thanks{Department of Mathematics, University of Wisconsin - Madison, 480 Lincoln Drive, 213 Van Vleck Hall, Madison, WI 53706. Email:\texttt{yiyu.wang@wisc.edu}}
}

\date{\today}

\begin{document}
\maketitle
\begin{abstract}
    We construct examples to show that having nef cotangent bundle is not preserved under finite ramified covers. Our examples also show that a projective manifold with Stein universal cover may not have nef cotangent bundle, disproving a conjecture of Liu-Maxim-Wang\cite{LiuMaximWang+2021}.
\end{abstract}

\section{Introduction}
 The notion of nefness of the cotangent bundle is closely related to various notions of hyperbolicity of a projective manifold (or more precisely, non-ellipticness), and it entails some topological obstructions. For example, any perverse sheaf on a projective manifold with nef cotangent bundle has nonnegative Euler characteristic \cite[Proposition 3.6]{LiuMaximWang+2021}. In \cite{LiuMaximWang+2021}, inspired by the Singer-Hopf conjecture and the Shafarevich conjecture, the authors made the following conjecture.
 
\begin{conjecture}(\cite[Conjecture 6.3]{LiuMaximWang+2021})\label{Main_Conj}
    Let $Y$ be a projective manifold. If the universal cover of $Y$ is a Stein manifold, then the cotangent bundle of $Y$ is nef.
\end{conjecture}

Since having Stein universal cover is preserved under finite ramified covering, one may expect the property of having nef cotangent bundle is also preserved under finite ramified covering. However, we show that it is not the case in this paper.
\begin{theorem}\label{Main_Theorem}
    For any positive integer $n\geq2$, there exists a finite surjective map between projective smooth $n$-folds $f: X\to Y$, such that $Y$ has nef cotangent bundle, while $X$ does not.
\end{theorem}

As a corollary, we give a counterexample to the conjecture above.

\begin{corollary}\label{Main_Corollary}
    There is a smooth projective variety $X$ whose universal cover is Stein while its cotangent bundle is not nef.
\end{corollary}

By a theorem of Kratz \cite[Theorem 2]{kratz1997compact}, any compact quotient of bounded domain in $\mathbb{C}^n$ (or any Stein manifold) has nef cotangent bundle. \cref{Main_Theorem} shows that boundedness in Kratz's theorem is necessary.

There are other hyperbolic-type properties that are preserved under finite surjective morphism, for example the property of having large fundamental group. A normal variety is said to have \emph{large fundamental group}, if its universal cover does not contain any positive dimensional proper complex subspaces (see for example \cite{Kollar1993}). If a variety has Stein universal cover, then it has large fundamental group (\cite[Proposition 6.7]{LiuMaximWang+2021}). Having large fundamental group is also related to various notions of hyperbolicity, as the Shafarevich conjecture predicts varieties with nonpositive sectional curvature have Stein universal covers (also see \cite{LiuMaximWang+2021}).

\begin{corollary}\label{Corollary_large_fundamental_group}
    There is a projective variety $X$ which has large fundamental group while its cotangent bundle is not nef.
\end{corollary}

The example in \cref{Main_Theorem} is constructed using the cyclic covering trick. We produce ramified covering map with smooth branched locus. Using a lemma of Sommese, we prove the ramified cover cannot have nef cotangent bundle. 

\subsubsection*{Acknowledgements.} The author would like to thank his advisor Botong Wang for suggesting this problem and helpful discussions. He also would like to thank Conner Simpson for polishing the first draft.

\section{A Lemma of Sommese}

A line bundle $L$ on a projective variety is called \textit{nef} if for every irreducible curve $C$, we have $L\cdot C\geq 0$. A vector bundle $E$ is called \textit{nef} if the tautological line bundle $\mathcal{O}_{\mathbf{P}(E)}(1)$ is nef on $\mathbf{P}(E)$.

The next lemma is due to Sommese \cite[Lemma 1.9]{Sommese1984}. Sommese used this lemma to classify all Hirzebruch surfaces with ample cotangent bundle.

\begin{lemma}\label{rami_locus_negative}
    Suppose $f:X\to Y$ is a double covering between $n$-dimensional smooth projective varieties over $\mathbb{C}$, $n\geq 2$, with smooth ramification locus $R\subset X$ and branched locus $B \subset Y$. Suppose furthermore at any point $p\in R$, $f$ can be written (analytical) locally as $(u_1,u_2,\cdots, u_n) \to (u_1^2,u_2,\cdots, u_n)$. If $X$ has nef cotangent bundle, then $R$ has nef conormal bundle.
\end{lemma}
\begin{proof}
    Consider the natural exact sequence
    \begin{equation*}
        f^* \Omega_Y^1 \to \Omega_X^1 \to \Omega_f^1 \to 0,
    \end{equation*}
    and restrict it to $R$,
    \begin{equation*}
        \Omega_X^1|_R \to \Omega_f^1|_R \to 0,
    \end{equation*}
    We claim that $\Omega_f^1|_R\simeq \mathcal{O}_R(-R)$. Assuming this is true, since $\Omega_X^1|_R$ is nef and any quotient bundle of a nef bundle is also nef, we must have $\mathcal{O}_R(-R)$ is nef.
    
    To prove the fact $\Omega_f^1|_R\simeq \mathcal{O}_R(-R)$, consider the natural short exact sequence
    \begin{equation}\label{fund_ses}
        0\to \mathcal{I}_R/\mathcal{I}_R^2 \to \Omega_X^1|_R \to \Omega_R^1\to 0.
    \end{equation}
    Let $\varphi$ be the composition $\mathcal{O}_R(-R)\simeq \mathcal{I}_R/\mathcal{I}_R^2 \to \Omega_X^1|_R \to \Omega_f^1|_R$, and we want to prove $\varphi$ is an isomorphism. This can be proved locally. In the coordinate patch $U$ that satisfies the second condition of this lemma, we know that $\Omega_f^1$ can be locally written as
    \begin{equation*}
        \frac{\bigoplus_{i=1}^n \mathcal O_U\{du_i\}}{\mathcal O_U\{u_1du_1\}\bigoplus_{i=2}^n \mathcal O_U\{du_i\}} \simeq \mathcal{O}_U/u_1\mathcal{O}_U\{du_1\},
    \end{equation*}
    which is a line bundle over $R\cap U$, generated by $du$. Moreover, the equivalence class of $u$ which generates $\mathcal{I}_R/\mathcal{I}_R^2$ maps to the class of $du$, which generates $\Omega_f^1|_R$. So $\varphi$ must be an isomorphism.
\end{proof}

\begin{corollary}\label{negative_intersection}
    Under the assumption of the above lemma, for any irreducible curve $C\subset R$, we must have $C\cdot R \leq 0$. In particular, if $X$ is of dimension 2, then $R\cdot R\leq 0$.
\end{corollary}
\begin{proof}
    For such a curve $C$, $R\cdot C=\mathrm{deg}_C(\mathcal{O}_X(R))=\mathrm{deg}_C(\mathcal{O}_R(R))\leq 0$ by $\mathcal{O}_R(-R)$ is nef.
\end{proof}

\begin{remark}
    What we actually proved here is the short exact sequence (\ref{fund_ses}) splits. This is clear geometrically: the tangent bundle $TX$ of $X$ restricted to $R$ has two natural subbundles, one is the tangent bundle of $R$, and the other is the kernel bundle of $df|_R:TX|_R\to TY|_B$. Moreover, the assumption ``analytic locally'' can be replaced by ``formal locally", since for checking $a$ is an isomorphism, we only need to verify it stalk by stalk formally. Therefore the argument works in all characteristics. 
\end{remark}

\section{The Cyclic Covering}
We recall the construction of cyclic covering.
\begin{proposition}[{\cite[Proposition 4.1.6]{Lazarsfeld2004}}]
Let $Y$ be an algebraic variety, and $L$ a line bundle on $Y$. Suppose given a positive integer $m \geq 1$ plus a non-zero section $$s \in \Gamma(Y,L^m)$$ defining a divisor $D \subset Y$. Then there exists a finite flat covering $\pi : X \to Y$, where $X$ is a scheme having the property that the pullback $L'= \pi^{*} L$ of $L$ carries a section $s'\in \Gamma(X,L')$ with $(s')^m = \pi^*s$. The divisor $D'= div(s')$ maps isomorphically to $D$. Moreover, if $Y$ and $D$ are non-singular, then so too are $X$ and $D'$.
\end{proposition}

The construction of $X$ can be described in detail locally. On an affine open subset $U$ of $Y$ that $L^m$ is trivial, $s$ can be viewed as a regular function on $U$. Then $\pi^{-1}(U)\subset Y \times \mathbf{A}^1$ is defined by the equation $t^m-s=0$. In particular, if $m=2$, and $D$ is nonsingular, then the covering we construct here satisfies the second condition of \cref{rami_locus_negative}. The reason is that analytic locally, the map $\pi$ can be written as $(x_1,x_2,\cdots,x_n,t) \to (x_1,x_2,\cdots,x_n)$. Since $D$ is nonsingular, at any given point $p\in D$, there exists some $i$ such that $(x_1,x_2,\cdots,\widehat{x_i},\cdots,x_n,s)$ is a local coordinate of $Y$. Then $(x_1,\cdots,\widehat{x_i},\cdots,x_n,t) \to (x_1,\cdots,\widehat{x_i},\cdots,x_n,s)$ is the desired coordinate presentation of $\pi$.

Now, we can prove the claimed results in the introduction.
\begin{proof}[Proof of \cref{Main_Theorem}]
    We pick any smooth $n$-fold $Y$ with nef cotangent bundle, a very ample line bundle $H$, and any smooth section $D\in |2H|$. By the construction above, we get a double covering map $f:X\to Y$, branched over $D$, satisfying all the conditions of \cref{rami_locus_negative}. In particular, if $X$ has nef cotangent bundle, then for any irreducible curve $C\subset D$, $D'\cdot f^*C=2D'\cdot f^{-1}(C) \leq 0$ by \cref{negative_intersection} ($f^{-1}(C)$ is still irreducible since $f$ maps $D'$ isomorphically to $D$). On the other hand, $D'\cdot f^*C=2D\cdot C=4H\cdot C>0$ by the assumption $H$ is very ample. Therefore $X$ cannot have nef cotangent bundle.
\end{proof}

\begin{proof}[Proof of \cref{Main_Corollary}]
    Let $Y$ in the above construction be an abelian variety. Then $Y$ has universal cover $\mathbb{C}^n$. Taking the fiber product in the category of complex spaces, we have the Cartesian diagram
    \[
    \begin{tikzcd}
        X' \arrow[d] \arrow[r] & \mathbb{C}^n \arrow[d] \\
        X \arrow[r]            & Y                     
    \end{tikzcd}
    \]
    where $X'$ is an infinite unramified cover of $X$. Since $X \to Y$ is finite, $X'\to \mathbb{C}^n$ is also finite, by \cite[Theorem V.1.1]{Grauert2004} $X'$ is Stein. Therefore, the universal cover of $X$ is Stein because any unramified covering (not necessarily finite) of a Stein manifold is Stein (see \cite{Stein1956}). But the cotangent bundle of $X$ is not nef.
\end{proof}

\begin{proof}[Proof of \cref{Corollary_large_fundamental_group}]
    For a normal variety $X$, having large fundamental group is equal to say for any positive dimensional cycle $w:W\to X$, the image of $\pi_1(W)\to\pi_1(X)$ is infinite (see \cite{Kollar1993}). It's obvious that having large fundamental group is preserved under finite surjective morphism using this definition. Take $Y$ to be an abelian variety, $X$ as in the proof of \cref{Main_Theorem}, then $X$ has large fundamental group while its cotangent bundle is not nef.
\end{proof}

Kobayashi hyperbolicity is an important example of hyperbolicity of complex manifolds. If a projective variety has ample cotangent bundle, then it also has Kobayashi hyperbolicity (see \cite[Theorem 6.3.26]{Lazarsfeld2004_2}). However, the above construction shows

\begin{corollary}
    There is a projective variety $X$ which is Kobayashi hyperbolic while its cotangent bundle is not nef.
\end{corollary}
\begin{proof}
    Take $Y$ to be any projective Kobayashi hyperbolic variety in the above construction. Then $f:X\to Y$ is finite surjective but $X$ does not have nef cotangent bundle. The variety $X$ is also Kobayashi hyperbolic, because if there is any nonconstant holomorphic map $g:\mathbb{C} \to X$, then $f\circ g$ is constant since $Y$ is Kobayashi hyperbolic. Therefore $g(\mathbb{C})$ is contained in a fiber of $f$, which is a finite set, so $g$ must be constant, and $X$ is also Kobayashi hyperbolic.
\end{proof}

\begin{remark}
    Such an example in dimension 2 was known a long time ago (see for example \cite{10.2307/2041670}). 
\end{remark}

All these examples show that, the nefness of cotangent bundle is too strong as a hyperbolicity condition. Inspired by \cite{arapura2021perverse}, one possible adjustment is to look at the class of projective manifolds that admit a finite morphism to a manifold with nef cotangent bundle. Since the pushforward of a perverse sheaf under a finite morphism is still perverse, by \cite[Proposition 3.6]{LiuMaximWang+2021} any perverse sheaf on such manifold has nonnegative Euler characteristic. In \cite{arapura2021perverse}, the authors proved the same property holds for projective varieties that admit a complex variation of Hodge structures with finite period map. This result also indicates that admitting a finite morphism to a projective manifold with nef cotangent bundle is a more appropriate hyperbolicity condition.

\printbibliography

\end{document}